\documentclass[11pt]{article}

\usepackage{amsmath,amssymb,amsthm}

\title{The Harmonic Descent Chain}

\author{%
  David J. Aldous\footnote{Department of Statistics,
 367 Evans Hall \#\  3860,
 U.C. Berkeley CA 94720;
    aldousdj@berkeley.edu;  www.stat.berkeley.edu/users/aldous.}
    \and Svante Janson\footnote{Department of Mathematics, Uppsala University, P.O.Box 480, SE-751 06 Uppsala Sweden; 
   svante.janson@math.uu.se} 
  \and 
   Xiaodan Li\footnote{Department of Mathematics,
Shanghai University of Finance and Economics,
China; lixiaodan@mail.shufe.edu.cn.}}
     
\begin{document}
     
\maketitle



\begin{abstract}
The decreasing Markov chain on \{1,2,3, \ldots\} with transition probabilities $p(j,j-i) \propto 1/i$ 
arises as a key component of the analysis of the beta-splitting random tree model.
We give a direct and almost self-contained ``probability" treatment of its occupation probabilities,
as a counterpart to a more sophisticated but perhaps opaque derivation using a limit continuum tree structure and Mellin transforms.
\end{abstract}

\newcommand{\ABS}[1]{\left(#1\right)} 
\newcommand{\veps}{\varepsilon} 

\newcommand{\Ex}{\mathbb{E}}
\renewcommand{\Pr}{\mathbb{P}}
\newcommand\qw{^{-1}}
\newcommand\qww{^{-2}}
\newcommand\qq{^{1/2}}
\newcommand\qqw{^{-1/2}}
\newcommand\gl{\lambda}
\newcommand\set[1]{\ensuremath{\{#1\}}}
\newcommand\floor[1]{\lfloor#1\rfloor}
\newcommand\asto{\overset{\mathrm{a.s.}}{\longrightarrow}}
\newcommand\pto{\overset{\mathrm{p}}{\longrightarrow}}

\newcommand\marginal[1]{\marginpar[\raggedleft\tiny #1]{\raggedright\tiny#1}}
\newcommand\SJ{\marginal{SJ} }
\newcommand\SJm[1]{\marginal{SJ: #1} }

\newtheorem{theorem}{Theorem}[section]	
\newtheorem{proposition}[theorem]{Proposition}
\newtheorem{lemma}[theorem]{Lemma}
\newtheorem{corollary}[theorem]{Corollary}
\newtheorem{claim}[theorem]{Claim}
\newtheorem{conjecture}[theorem]{Conjecture}
\newtheorem{example}[theorem]{Example}
\newtheorem{require}[theorem]{Requirement}
\newtheorem{fact}[theorem]{Fact}
\newtheorem{definition}[theorem]{Definition}

\numberwithin{equation}{section}

\section{Introduction}
Write    $h_n := \sum_{i=1}^n \frac1i $ for the harmonic series.
We will study the discrete-time Markov chain 
$(X_t, t = 0,1,2,\ldots)$ on states $\{1,2,3 \ldots\}$ with transition probabilities 
 \begin{equation}
  p(j,i) = \frac{1}{(j-i) h_{j-1}},  \ 1 \le i < j, \ j \ge 2 
  \label{P:def}
  \end{equation}
 and $p(1,1) = 1$. 
 So sample paths are strictly decreasing until absorption in state $1$. 
 The simple form (\ref{P:def}, \ref{L:def}) of the transitions suggests that this chain might have been arisen previously in some different context, but we have not found any reference.
  Let us call this the {\em harmonic descent} (HD) chain.

As discussed at length elsewhere \cite{beta2},  
the HD chain arises in a certain model of random $n$-leaf trees: the chain describes the number of descendant leaves of a vertex, as one moves 
along the path from the root to a uniform random leaf.
 In this article we study the ``occupation probability", that is
 \begin{equation}
 \mbox{
 $a(n,i) := $ probability that the chain started at state $n$ is ever in state $i$.
 }
 \label{def:ani}
 \end{equation}
 So $a(n,n) = a(n,1) = 1$.
 The motivation for studying $a(n,i)$ is that
 the mean number of $i$-leaf subtrees of a random $n$-leaf tree equals $n a(n,i)/i$.
 There is a general notion of the {\em fringe distribution} \cite{me-fringe,fringe} of a tree as viewed from a  random leaf. 
 Knowing the explicit value of the limits  $\lim_{n \to \infty} a(n,i)$
 enables one in \cite{beta3} to describe explicitly the fringe distribution in our tree model. 
 
 We should also mention that $\sum_{i=2}^n a(n,i)$ is just the mean time $\Ex_n T_1$ for the chain started at $n$
 to be absorbed in state $1$.
 A detailed analysis of the distribution of $T_1$ and related quantities, by very different methods, is 
 given in \cite{beta1}.
 
It seems very intuitive (but not obvious at a rigorous level) that the limits $\lim_{n \to \infty} a(n,i)$ exist, 
however there seems no intuitive reason to think there should be some simple formula for the limits.  
But the purpose of this paper is to give an almost\footnote{We quote one sharp estimate from \cite{beta1} as our Theorem \ref{T:beta1}.}  self-contained proof of
 \begin{theorem}
  \label{T:alimit}
 For each $i = 2,3,\ldots$,
 \begin{eqnarray}
\lim_{n \to \infty} a(n,i) &=& \frac{6 h_{i-1}}{\pi^2 (i-1)} .\label{talimit}
 \end{eqnarray}
 \end{theorem}
 As the reader might guess, we will encounter {\em Euler's formula}  $\sum_{i \ge 1} 1/i^2 = \pi^2/6$.
 
We prove Theorem \ref{T:alimit}  in two stages.
 In section \ref{sec:proofalim}  we will prove by coupling that the limits $a(i)$ exist. 
 This is straightforward in outline, though somewhat tedious in detail.
 More interesting, and therefore presented first in section \ref{SSpf1}, is the explicit formula for the limits $a(i)$.
  The limits satisfy an infinite set of equations \eqref{aisum}, and a solution was found by inspired guesswork.
  Then we need only to check that the solution is unique.
  
 Could one prove Theorem \ref{T:alimit} without guesswork? 
 We do not know how to do so within the Markov chain setting.
 In the random tree model, there is a limit continuum tree structure within which there is a continuous analog 
 of the ``hypothetical" chain described below, and by analysis of that process and the exchangeability properties of the
 continuum tree, forthcoming work \cite{beta3} relates the 
 $(a(i))$ to the $x \downarrow 0$ behavior of a certain function $f(x)$ determined by its Mellin transform:
 by technically intricate analysis one can re-prove Theorem \ref{T:alimit} via a ``proof by calculation".

 \section{The explicit limit}\label{SSpf1}
 In section \ref{sec:proofalim}  we will prove by coupling
\begin{proposition}
\label{P:alimit}
 For each $i = 1,2,3,\ldots$,
 \begin{eqnarray}\label{apr3}
\lim_{n \to \infty} a(n,i) &:=& a(i) \mbox{ exists}  \label{ailim} \\
a(i)  &=&  \sum_{ j >i} a(j) p(j,i)  \label{aisum} 
 \end{eqnarray}
 and
 \begin{equation}
 a(1) = 1 .  
 \label{aisum2} 
 \end{equation}
\end{proposition}
 Such a coupling proof does not give any useful quantitative information about the limits $a(i)$. 
 In this section we show how to derive the value of the limits, granted Proposition \ref{P:alimit}.
 
For fixed $n\ge 2$ it is clear that the values of 
$(a(n,i), 1\le i \le n)$ are determined by the natural system of equations
\begin{eqnarray}
 a(n,n) &=&  1, \label{ann}\\
a(n,i)  &=&  \sum_{n \ge  j >i} a(n,j) p(j,i), \quad n - 1 \ge i \ge 1, \label{ansum} \\
a(n,1) &=& 1 . \label{an1}
\end{eqnarray}
So the equations 
\eqref{aisum}--\eqref{aisum2}
are what one expects as the $n \to \infty$ analog of  \eqref{ansum}--\eqref{an1}.
It is not obvious that these equations have a unique solution, but the following solution was
 found by  inspired guesswork, rather than calculation.
 Define
 \begin{align}\label{bi}
 b(1) = 1; \quad b(i) =  \frac{6 h_{i-1}}{\pi^2 (i-1)} , \quad i \ge 2.
\end{align} 
 
  \begin{lemma}
 \label{L:guess}
 Equations \eqref{aisum}--\eqref{aisum2} are satisfied by $b(i)$ in \eqref{bi}.
 \end{lemma}
  \begin{proof}
 Equation \eqref{aisum} for $b(i)$ is explicitly
 \begin{equation}
 b(i) = \sum_{j>i} \frac{b(j)}{h_{j-1} (j-i)} , \quad  i \ge 1 .
 \label{bj}
 \end{equation}
 To verify \eqref{bj}, consider $i >1$. We have
 \begin{align}\label{apr1}
 \sum_{j>i} \frac{b(j)}{(j-i)h_{j-1}} 
 &=
 \frac{6}{\pi^2}  \sum_{j>i} \frac{1}{(j-i)(j-1)}
\notag\\
  &=
 \frac{6}{\pi^2(i-1)}  \sum_{j>i} \left( \frac{1}{j-i} - \frac{1}{j-1} \right)
\notag\\
   &=
 \frac{6}{\pi^2(i-1)} \lim_{k \to \infty} \left( \sum_{j=1}^{k-i}
     \frac{1}{j} -  \sum_{j=i}^{k-1} \frac{1}{j} \right)
\notag\\
    &=
 \frac{6}{\pi^2(i-1)} \ h_{i-1} = b(i) .
 \end{align}
 For $i = 1$:
 \[ 
  \sum_{j>1} \frac{b(j)}{h_{j-1} (j-1)} = \frac{6}{\pi^2} \sum_{j>1} \frac{1}{(j-1)^2} = 1 = b(1) .
  \]
 \end{proof}
 As noted before, at first sight we do not know that these  equations \eqref{aisum}--\eqref{aisum2} have a 
{\em unique}  solution.
We need a further careful argument to prove that $a(i)\equiv b(i)$,which then (granted Proposition \ref{P:alimit}) 
completes a proof of Theorem \ref{T:alimit}.

\begin{proposition}
\label{P:ablimit}
For $a(i)$ defined by the limit \eqref{apr3} and $b(i)$ defined by \eqref{bi}, we have
$a(i) = b(i), i \ge 1$.
\end{proposition}
\begin{proof}
Fix large $k$.
By considering the chain started at $n$ and decomposing at the jump over $k$:
\[
a(n,i) = \sum_{n \ge m > k} \sum_{k \ge j \ge i} a(n,m) p(m,j)a(j,i), \ i \le k < n .
\]
Letting $n \to \infty$ suggests
\begin{equation}
a(i) = \sum_{m > k} \sum_{k \ge j \ge i} a(m) p(m,j)a(j,i), \ i \le k .
\label{aii}
\end{equation}
This does not have a direct Markov chain interpretation, because one cannot
formalize the idea of starting a hypothetical version of the chain from
$+\infty$ at time  
$- \infty$  and making its occupation probability be $(a(i))$.
Nevertheless \eqref{aii} is correct, and can be derived from the relation
\eqref{aisum}.
Fix $i$ and $k$ with $1\le i\le k$ 
and in \eqref{aisum} recursively substitute $a(j)$ by the same equation for
every $j\le k$.
This yields, with summation over all paths of the form
$\eta = (i_1 = m > i_2 = j> i_3 > \ldots > i_q > i)$ with $i_1 > k \ge i_2$:
\[ a(i) = \sum_\eta a(m) p(m,j)p(j,i_3)p(i_3,i_4) \dotsm p(i_q,i) \]
which collapses to \eqref{aii}.
 The key part of our proof is that the argument above for \eqref{aii} depends only on $(a(i))$ satisfying  \eqref{aisum}, 
 which by Lemma \ref{L:guess} also holds for $(b(i))$, and so the conclusion holds also for $(b(i))$:
\begin{equation}
b(i) = \sum_{m > k} \sum_{k \ge j \ge i} b(m) p(m,j)a(j,i), \ i \le k .
\label{bdec}
\end{equation}
Now consider
\begin{align}\label{apr2}
 \hat{b}_k(j) := \sum_{m>k} b(m)p(m,j), \ 1 \le j \le k,   
\end{align}
motivated as the overshoot (over $k$) distribution
associated with our hypothetical version of a chain run from time $-\infty$. 
We will verify later
 \begin{lemma} 
 \label{L:hat}
$ (\hat{b}_k(j) , 1 \le j \le k)$ is indeed a probability distribution, for each $k \ge 1$.
 \end{lemma}
 Now let us re-write \eqref{bdec} as
\begin{equation}
b(i) =  \sum_{k \ge j \ge i}  \hat{b}_k(j)  a(j,i)
=  \sum_{k \ge j \ge 1}  \hat{b}_k(j)  a(j,i)
, \quad i \le k 
\label{bdec2}
\end{equation}
and thus
\begin{equation}
b(i)-a(i) =  \sum_{k \ge j \ge 1}  \hat{b}_k(j)  (a(j,i)-a(i)), 
\quad i \le k 
\label{bdec2'}
.\end{equation}
It follows from \eqref{apr2} that $\hat b_k(j)\to0$ as $k \to \infty$ for every
fixed $j$,
and thus the distributions $\hat{b}_k$ go off to infinity, that is 
\[ \lim_{k \to \infty} \hat{b}_k[1,L] = 0 \mbox{ for each } L < \infty . \]
So letting $k \to \infty$ for fixed $i$ we see that \eqref{bdec2'} and
Proposition \ref{P:alimit} imply that $b(i) = a(i)$, 
establishing Proposition \ref{P:ablimit}.
  \end{proof}
  
  \begin{proof}[Proof of Lemma \ref{L:hat}]
 
  From the definitions of $b(m)$ and $p(m,j)$ this reduces to proving that for each $k \ge 2$
\begin{equation}
\sum_{j=1}^k \sum_{m=k+1}^\infty \frac{1}{(m-j)(m-1)} = \frac{\pi^2}{6} .
\label{reduced}
\end{equation}
For $k = 1$ this is Euler's formula.
Furthermore, for $k \ge 1$, we have,
writing $s_k$ for the left side of \eqref{reduced} and $\phi(m,j) =
\frac{1}{(m-j)(m-1)}$, and arguing similarly as in \eqref{apr1},
\begin{eqnarray*}
 s_{k+1} - s_k &=&  
\sum_{m=k+2}^\infty  \phi(m,k+1) - \sum_{j=1}^k \phi(k+1,j) \\
&=& 
\sum_{m=k+2}^\infty  \frac{1}{(m-k-1)(m-1)} - \sum_{j=1}^k  \frac{1}{(k+1-j)k}\\
&=&  \frac{1}{k} \sum_{m=k+2}^\infty  \left( \frac{1}{m-k-1} - \frac{1}{m-1} \right) - \frac{1}{k} h_k\\
&=& \frac{1}{k} h_k - \frac{1}{k} h_k 
=0
\end{eqnarray*}
establishing \eqref{reduced}.
\end{proof}

 \section{Proof of Proposition \ref{P:alimit}.}
 \label{sec:proofalim}
 \subsection{Outline of proof}
 As noted in other discussions of the random tree model \cite{beta2,beta1},
 it is often more convenient to work with the associated {\em continuous time} chain which holds in each state $i \ge 2$ 
 for an Exponential(rate $h_{i-1}$) time, in other words has transition rates
 \begin{equation}
  \lambda(j,i) = \tfrac{1}{j-i}, \ \  1 \le i < j < \infty. 
  \label{L:def}
  \end{equation}
  We call this the {\em continuous HD chain} $(X_t, 0 \le t < \infty)$.
  Switching to this continuous time chain does not affect our stated definition of $a(n,i)$, though of course the mean occupation time in state $i$ changes from $a(n,i)$ to $a(n,i)/h_{i-1}$.
 And switching  does more substantially change the absorption time (to state $1$) $T_1$.
 As mentioned before, a detailed study of absorption time distributions is given in \cite{beta1}, and we quote two results from there.

 \begin{theorem}[\cite{beta1} Theorem 1.1]
 \label{T:beta1}
 \begin{equation} 
 \mbox{In continuous time, } 
 \Ex_x [T_1] =  \tfrac{6}{\pi^2} \log x + O(1) \mbox{ as } x \to \infty .
 \label{ED}
 \end{equation}
  \end{theorem}
 
 \begin{lemma}[From \cite{beta1} Corollary 1.3]
 \label{L:Tk}
  Let, for positive integers $k$, 
\begin{align}
T_k:=\min\{t:X_t\le k\}.  
\end{align}
Then, for any fixed $k\ge1$ and $t\ge0$, we have 
$\Pr_x(T_k\ge t)\to 1$ as $x \to \infty$; in other words,
$T_k\pto\infty$.  
\end{lemma}
Lemma \ref{L:Tk} can in fact be obtained by a simple direct proof, given for completeness in section \ref{sec:prlemma}.

 So now we consider the continuous time setting. 
 We will use a {\em shift-coupling}  \cite{thorisson}. 
 For our purpose, a shift-coupling
  ($(X_t,Y_t), 0 \le t < \infty)$ started at ($x_0,y_0)$ is a  process such that, conditional on $X_t = x_t,Y_t = y_t$ and the past, either \\
  (i) over $(t,t+dt)$ each component moves according to $\lambda(\cdot,\cdot)$, maybe dependently; or \\
 (ii) one component moves as above while the other remains unchanged.
 
 \smallskip \noindent
 Such a process must reach state $(1,1)$ and stop, at some time $T_{(1,1)} < \infty$.
 So the {\em coupling time} is such that
 \[ T^{couple} := \min\{t: X_t = Y_t\} \le T_{(1,1)}    \]
 and we can arrange that $X_t  = Y_t$ for $t \ge T^{couple}$.
 Write $S^{couple}:= X_{T^{couple}}= Y_{T^{couple}}$ for the coupling {\em state}.
 We will construct a shift-coupling in which, for each $i \ge 1$,
 \begin{equation}
 \Pr_{x_0,y_0} (S^{couple} < i) \to 0 \mbox{ as } x_0, y_0 \to \infty. 
 \label{sec-c}
 \end{equation}
 This is clearly sufficient to prove the main ``limits exist" part \eqref{ailim}
 of Proposition \ref{P:alimit},
 because $|a(x_0,i) -a(y_0,i)| \le \Pr_{x_0,y_0} ( S^{couple}< i)$.
 
  In outline the construction is very simple.  
 If the initial states $x_0$ and $y_0$ are not of comparable size, then we run the chain only from the larger state (as in (ii) above) until they are of comparable size;
 then at each time there is some non-vanishing probability that we can couple at the next transition (as in (i) above).
 
 The details are given via two lemmas below.
By symmetry, it suffices to consider the case $x_0\le y_0$, and
by considering subsequences we may assume that $x_0/y_0\to a$ for some
$a\in [0,1]$.  
First we consider the ``comparable size starts" case $a>0$, and then the
case $a=0$.
 As in \cite{beta2} and as suggested by \eqref{ED} we analyze the processes
 on the log scale.
 
  \subsection{The maximal coupling regime}\label{SSmaximal}
 In the maximal coupling regime, we construct the joint process $((X_t,Y_t), 0 \le  t < \infty)$ as follows.
 From $(X_t,Y_t) = (x,y)$ with $x\le y$, 
 each component moves according to $\lambda(\cdot,\dot)$ but with the joint distribution that maximizes the probability that they move to the same state.
That joint distribution is such that,
for infinitesimal $dt$,
 \[ \Pr_{x,y}(X_{dt} = Y_{dt} = i) = \Pr_{x,y}(Y_{dt} = i) = \frac{dt}{y-i}, 
\quad 1 \le i < x .\]
So
 \[ \Pr_{x,y}(X_{dt} = Y_{dt}) =  \sum_{i=1}^{x-1} \frac{dt}{y-i} 
\ge \frac{x-1}{y}.
\] 
Hence, for any $c \in (0,1)$, if $c\le x/y \le 1$, then
\begin{equation}
\Pr_{x,y}(X_{dt} = Y_{dt})  
\ge  (c-y\qw) dt 
\label{one_step}
.\end{equation}

 \begin{lemma}
For the maximal coupling process,
if $y_0 \to \infty$ and $x_0/y_0 \to a \in (0,1]$ then for each $k$ we have
$\Pr_{x_0,y_0}(S^{couple}\le k) 
\to 0 $.
\label{L:maxcoup}
\end{lemma}
\begin{proof}
Write $T^Y_k := \min\{t: Y_t \le k\}$,
and note that
$\set{S^{couple}\le k} = \set{ T^Y_k \le T^{couple}}$.
 Consider the process $(X_t,Y_t)$ using the maximal coupling with $(X_0,Y_0)
= (x_0,y_0)$, where $x_0 \le y_0$.
The coupling is stochastically monotone, so  $X_t \le Y_t$ and the absorption times into state $1$ satisfy $T^X_1 \le T^Y_1$.
 By  \eqref{ED} we  have
 \begin{equation}
 \Ex_{x_0,y_0}  [T^Y_1  - T^X_1] 
=  \frac{6}{\pi^2} \log \frac{y_0}{x_0} + O(1)
 \label{DXDY}
 \end{equation}
 where the $O(1)$ bound is uniform in all $x_0,y_0\ge1$.
 
 Fix $0 < c < a \le1$. 
 Consider the stopping time
 \[ U_c := \min \{t: X_t/Y_t \le c \} . \]
Fix also a large $\tau$ and let $\ell\ge k$. 
Then
 \begin{align}\label{pask1}
&    \Pr_{x_0,y_0}(S^{couple}  \le k)
=     \Pr_{x_0,y_0}(T^{couple}\ge T^Y_k)
\le    \Pr_{x_0,y_0}(T^{couple}\ge T^Y_\ell)
\notag\\& \qquad
\le   
\Pr_{x_0,y_0}(U_c< T^Y_\ell) 
+
\Pr_{x_0,y_0}(T^Y_\ell\le\tau) 
+
\Pr_{x_0,y_0}(T^{couple}\wedge U_c\wedge  T^Y_\ell>\tau) 
. \end{align}
We consider the three terms in \eqref{pask1} separately.

On the event $\{ U_c  < T^Y_\ell  \}$ 
the conditional expectation of $(T^Y_1 - T^X_1)$ is at least  
$ \tfrac{6}{\pi^2} \log 1/c - O(1) $ by (\ref{DXDY}) and conditioning on
$(X_{U_c},Y_{U_c})$, noting that $Y_{U_c}/X_{U_c}\ge1/c$.
So by Markov's inequality and (\ref{DXDY}) again,
there exists a constant $C$ (not depending on $a,c$)
such that, provided $\log 1/c > C$ and $y_0$ is large enough,
\begin{align}
 \Pr_{x_0,y_0} (U_c \le T^Y_k) 
\le \frac{ \log (y_0/x_0) + O(1)}{\log 1/c  - O(1)}
\le \frac{ \log 1/a + C}{\log 1/c  - C}. 
\end{align}
This holds for any sufficiently small $c > 0$, and
 can be made arbitrarily small  by choosing $c$ small.

As long as $\{t < T^{couple}\wedge U_c\wedge T^Y_\ell \}$, 
the coupling event happens at
rate $ \ge c - \ell\qw$ by \eqref{one_step}, so
\begin{align}
  \Pr_{x_0,y_0}(T^{couple}\wedge U_c\wedge T^Y_\ell>\tau) \le
\exp\left( (-c + \ell\qw ) \tau \right),
\end{align}
which for fixed $c$ can be made arbitrarily small by choosing 
$\ell\ge 2/c$ and 
$\tau$ large.

Finally, for fixed $\ell$ and $\tau$,
$\Pr_{x_0,y_0}(T^Y_\ell\le\tau) \to0$ as $y_0\to\infty$ by Lemma \ref{L:Tk}.

Consequently, \eqref{pask1} shows that
$\Pr_{x_0,y_0}(S^{couple}  \le k)\to0$ as $y_0\to\infty$ and $x_0/y_0\to a$.
\end{proof}

 \subsection{The large discrepancy regime}\label{SSshift}
 Given Lemma \ref{L:maxcoup}, to establish \eqref{sec-c} it remains only to consider the case $x_0/y_0 \to 0$.
Here we first run the chain $(Y_t)$ starting from $y_0 \gg x_0$ while
holding $X_t = x_0$ fixed.
 The next lemma shows that the $(Y_t)$ process does not overshoot $x_0$ by
 far, on the log scale. 
 
 For $x\ge1$  write as above
 \[ T^Y_x := \min\{t: Y_t \le x\}.  \]
Let also
 \[ V_x := \log x - \log Y_{T^Y_x}  \quad \mbox{(the overshoot)}.\]
 \begin{lemma}
 There exists an absolute constant $K$ such that
 \[ \Ex_y [V_x] \le K, \quad 1 \le x < y < \infty  . \]
 \label{L:over}
 \end{lemma}
 
  \begin{proof}
 Here we work in  discrete time, which does not change  $\Ex_y [V_x] $.
 Consider a single transition $y \to Y_1$.
 From the transition probabilities \eqref{P:def} we obtain the exact
 formula, for any $a>0$,  and writing $b=e^{-a}\in(0,1)$,
 \begin{align}\label{paskris}
   \Pr_y(\log y - \log Y_{1} \ge a) 
=
\Pr_y(Y_1 \le e^{-a}y)
=
\frac{\sum_{y-\floor{b y}}^{y-1}\frac{1}i}{h_{y-1}}
=
\frac{h_{y-1}-h_{y-\floor{b y}-1}}{h_{y-1}}
. \end{align}
 If, say, $a\ge1$ is fixed, then it follows by the formula 
 \begin{equation}
h_n = \log n + \gamma + o(1) \mbox{ as } n \to \infty
 \label{hn}
 \end{equation}
that as $y \to \infty$
 \begin{align}
   \Pr_y(\log y - \log Y_{1} \ge a) 
\to
 \frac{\theta[a, \infty)}{\log y}
 \label{approx}
, \end{align}
 where $\theta$ is the measure on $(0, \infty)$ defined by
 \begin{equation}
  \theta[a, \infty) := - \log (1 - e^{-a}) . 
  \label{def:theta}
  \end{equation}
  We will show that the approximation \eqref{approx} holds within some constant factors
uniformly for all $y\ge2$ and 
$a\in[a_1(y),a_2(y)]$ where
$a_1(y):= \log y-\log(y-1)=-\log(1-1/y)$
and $a_2(y):=\log y$.
(Note that for $a\le a_1(y)$, trivially $\Pr_y(\log y-\log Y_1\ge a)=1$,
and for $a>a_2(y)$, trivially 
$\Pr_y(\log y-\log Y_1\ge a)=0$.) 
That is, for some $C_1,C_2>0$,
 \begin{align}
C_1 \frac{\theta[a, \infty)}{\log y}
\le
   \Pr_y(\log y - \log Y_{1} \ge a) 
\le 
C_2 \frac{\theta[a, \infty)}{\log y},
\qquad a\in [a_1(y),a_2(y)]
 \label{pask}
. \end{align}
To verify \eqref{pask}, note first that by \eqref{hn} we only have to
estimate the numerator in \eqref{paskris}.
We have  
\begin{align}\label{paskagg}
\int_{y-\floor{by}}^{y}\frac{dx}{x}\le
\sum_{y-\floor{b y}}^{y-1}\frac{1}i 
\le \int_{y-\floor{by}}^{y-1}\frac{dx}{x} +\frac{1}{y-\floor{by}}
\end{align}
where $b=e^{-a}\in[e^{-a_2(y)},e^{-a_1(y)}]=[\frac{1}{y},\frac{y-1}{y}]$,
and it is easily seen that both sides of \eqref{paskagg} are within
constant factors of $-\log(1-b)=\theta[a,\infty)$, thus showing \eqref{pask}.

The measure $\theta$ has the property 
\begin{align}
  \label{pingst}
  \frac{ \int_a^\infty (u-a) \theta(du) }{\theta[a,\infty)} \uparrow 1
  \mbox{ as } a \uparrow \infty,
\end{align}
and thus the conditional mean excess over $a$ is bounded above by 1.
One can now use \eqref{pask} to see that
there exists $K  < \infty$ 
such that the corresponding property holds for one transition $y \to Y_1$ of
the discrete chain on the log scale: 
 \begin{equation}
  \Ex_y[ (\log y - \log Y_{1} - a)^+] 
\le K \Pr_y(\log y - \log Y_{1} \ge a) .
 \label{Eyyy}
 \end{equation}
In fact, for $a_1(y)\le a\le a_2(y)$, this follows immediately from
\eqref{pask}, \eqref{pingst},
and
 \begin{equation}
  \Ex_y[ (\log y - \log Y_{1} - a)^+] 
=\int_a^{a_2(y)}\Pr_y(\log y - \log Y_{1} >s)ds
 \label{Eyyy+}
. \end{equation}
For $a<a_1(y)$, \eqref{Eyyy} follows from the case $a=a_1(y)$, with the 
right side equal to K, and for $a>a_2(y)$, both sides of \eqref{Eyyy} are 0.

 Now fix $x$.  
 For $y \ge x$ write
   \[ m(x,y) := \max_{x \le z \le y} \Ex_z[V_x]  \]
and note that $m(x,x)=0$.
If $y>x$,  by considering  the first step $y \to Y_1$, which 
   goes either into the interval $[1,x]$ (probability $q_{x,y}$ say) or into
   the interval $[x+1,y-1]$:
   \begin{align}\label{paska}
 \Ex_{y}[V_x] \le    \Ex_{y}[  (\log x - \log Y_1)^+  ]    + \Pr_y( \log x - \log Y_1 < 0    )  m(x,y-1) .     
   \end{align}
 By \eqref{Eyyy} with $a = \log y - \log x$
 \begin{align}\label{paskb}
   \Ex_{y}[  (\log x - \log Y_1)^+  ]    
\le  K \Pr_y(\log x - \log Y_1 \ge 0) 
=Kq_{x,y}
. \end{align}
 Combining \eqref{paska} and \eqref{paskb}:
 \begin{align}\label{paskc}
 \Ex_{y}[V_x] \le   K q_{x,y} + (1-q_{x,y})m(x,y-1)   
 \end{align}
  and so the bound $m(x,y) \le K$ holds
 by induction on $y = x, x+1, x+2, \ldots$.
\end{proof}

 \subsection{Completing the proof of Proposition \ref{P:alimit}.}
  Given Lemma \ref{L:maxcoup}, 
to show \eqref{sec-c}, 
it remains only to consider the case $x_0/y_0 \to 0$.
  Use the shift regime dynamics $(x_0, Y_t)$ in section \ref{SSshift}
from $Y_0 = y_0$ until time
   \[ T^Y_{x_0} := \min\{t: Y_t \le x_0\}  ,\]
and then use the maximal coupling in section \ref{SSmaximal}. 
   Lemma \ref{L:over} shows that the overshoot of the first phase
 \[ V_{x_0} := \log x_0 - \log Y_{T^Y_{x_0} } \]
 has $\Ex_{y_0}[V_{x_0}] \le K$.
Hence the overshoots are tight, and by considering a subsequence we may
assume that $V_{x_0}$ converges in distribution to some random variable $V$.
For convenience, we may also by the
Skorohod coupling theorem \cite[Theorem~4.30]{Kallenberg}
assume that $V_{x_0}$ converges to $V$ almost surely, and thus
\begin{align}
  x_0/Y_{T^Y_{x_0}} \asto e^V>0.
\end{align}
This allows us to 
condition on $Y_{T_{x_0}}$ and
apply Lemma \ref{L:maxcoup} 
(with $X$ and $Y$ interchanged)
with starting state $(x_0, Y_{T_{x_0} })$.
Hence, for every fixed $k$, 
we have
\begin{align}
\Pr_{x_0,y_0}\left(S^{couple}\le k\mid Y_{T_{x_0}}\right)\to 0 \qquad \text{a.s.}  
\end{align}
and thus \eqref{sec-c} follows by taking the expectation.
As noted after \eqref{sec-c}, this completes the proof of 
the main ``limits exist" part \eqref{ailim}
 of Proposition \ref{P:alimit}.
 To complete the proof of
Proposition \ref{P:alimit},
and by  section \ref{SSpf1} thus also the proof of Theorem \ref{T:alimit},
we need to verify \eqref{aisum}, that is 
\begin{equation}
a(i)  =  \sum_{ j >i} a(j) p(j,i), \quad i \ge 1 .  \label{aisum3} 
\end{equation}
From the pointwise convergence $a(n,j) \to a(j)$, Fatou's lemma gives
\begin{eqnarray*}
\sum_{j>i} a(j)p(j,i) &\le& \liminf_n \sum_{j>i} a(n,j)p(j,i) \\
&= & \lim_n a(n,i) = a(i) .
\end{eqnarray*}
To prove equality we need to show
that, for every fixed $i$, 
\begin{equation}
\lim_{L \to \infty} \liminf_n \sum_{j>L} a(n,j)p(j,i) = 0.
\label{Lan}
\end{equation}
This is an easy consequence of the overshoot bound, as follows.
Assume, as we may, $L>i$. 
The sum above is the probability, starting at $n$, that the first entrance into $\{L, L-1, L-2, \ldots\}$ is at $i$, 
which is bounded by the probability that first entrance is in $\{i, i-1, i-2, \ldots\}$.
The latter, in the notation of Lemma \ref{L:over}, is just 
$\Pr_n(V_L \ge \log L - \log i)$.
By Markov's inequality and Lemma \ref{L:over}
\[  \sum_{j>L} a(n,j)p(j,i) \le \frac{K}{\log L - \log i} \]
implying \eqref{Lan}.

  \section{Direct proof of Lemma \ref{L:Tk}}
  \label{sec:prlemma}
  Let $Z_t:=X_t\qqw$. Then \eqref{L:def} implies that for an infinitesimal
time $dt$, we have
\begin{align}
  \Ex_x[Z_{dt}]-x\qqw &
= dt\sum_{1\le i<x}\gl(x,i)\left(\frac{1}{\sqrt i}-\frac{1}{\sqrt x}\right)
= dt\sum_{1\le i<x}\frac{\sqrt x-\sqrt i}{(x-i)\sqrt{i}{\sqrt x}}
\notag\\&
= dt\sum_{1\le i<x}\frac{1}{(\sqrt x+\sqrt i)\sqrt{i}{\sqrt x}}
\le \frac{ dt}{x}\sum_{1\le i<x}\frac{1}{\sqrt{i}}
\notag\\&
\le \frac{ dt}{x}\int_{0}^x \frac{1}{\sqrt y}\, dy
=2\frac{dt}{\sqrt{x}}=2 Z_0 dt.
\end{align}
Hence, by the Markov property,
\begin{align}
  \frac{d}{dt}\Ex_x[Z_t] \le 2 Z_t,
\qquad t\ge0,
\end{align}
and consequently,
\begin{align}
  \Ex_x [Z_t] \le e^{2t} Z_0 = e^{2t}x\qqw,
\qquad t\ge0.
\end{align}
Finally, Markov's inequality yields
\begin{align}
  \Pr_x (T_k \le t) 
= \Pr_x(X_t\le k)
= \Pr_x(Z_t\ge k\qqw)
\le e^{2t}\sqrt{k/x}.
\end{align}

\providecommand{\bysame}{\leavevmode\hbox to3em{\hrulefill}\thinspace}
\providecommand{\MR}{\relax\ifhmode\unskip\space\fi MR }
\providecommand{\MRhref}[2]{%
  \href{http://www.ams.org/mathscinet-getitem?mr=#1}{#2}
}
\providecommand{\href}[2]{#2}


\begin{thebibliography}{1}

\bibitem{me-fringe}
David Aldous, \emph{Asymptotic fringe distributions for general families of
  random trees}, Ann. Appl. Probab. \textbf{1} (1991), no.~2, 228--266.
  \MR{1102319}

\bibitem{beta2}
\bysame, \emph{The critical beta-splitting random tree model {II}: Overview and
  open problems}, 2023, arxiv.org/abs/2303.02529.

\bibitem{beta3}
David Aldous and Svante Janson, \emph{The critical beta-splitting random tree
  model {III}}, 2024, In preparation.

\bibitem{beta1}
David Aldous and Boris Pittel, \emph{The critical beta-splitting random tree
  {I}: Heights and related results}, arXiv:2302.05066, 2023.

\bibitem{fringe}
Cecilia Holmgren and Svante Janson, \emph{Fringe trees, {C}rump-{M}ode-{J}agers
  branching processes and {$m$}-ary search trees}, Probab. Surv. \textbf{14}
  (2017), 53--154. \MR{3626585}

\bibitem{Kallenberg}
Olav Kallenberg, \emph{Foundations of modern probability}, second ed.,
  Probability and its Applications (New York), Springer-Verlag, New York, 2002.
  \MR{1876169}

\bibitem{thorisson}
Hermann Thorisson, \emph{Shift-coupling in continuous time}, Probab. Theory
  Related Fields \textbf{99} (1994), no.~4, 477--483. \MR{1288066}

\end{thebibliography}

\end{document}